\documentclass[12pt,final,reqno]{amsart}
\usepackage{ntung}
\usepackage{apptools}
\usepackage{amsrefs}
\usepackage[table]{xcolor}    
\usepackage{graphicx}         
\usepackage{caption}          
\usepackage{adjustbox}        
\usepackage{tabularx}

\begin{document}

\title{Cutting a unit square and permuting blocks}
\author{Nathan Tung}
\address{Stanford University, CA 94305, USA}
\email{ntung@stanford.edu}

\begin{abstract}
    Consider a random permutation of $kn$ objects that permutes $n$ disjoint blocks of size $k$ and then permutes elements within each block. Normalizing its cycle lengths by $kn$ gives a random partition of unity, and we derive the limit law of this partition as $k,n \to \infty$. The limit may be constructed via a simple square cutting procedure that generalizes stick breaking in the classical case of random permutations ($k=1$). The expected size of the largest part of this square cutting distribution is approximated to be $0.40$, in contrast with the Golomb-Dickman constant around $0.624$ describing the longest cycle of a uniform random permutation as well as the largest prime factor of a random integer. The distribution function of this largest part is shown to also be the mean of a certain multiplicative function. Along the way we give the first extension of the Erd\H{o}s-Turán law to a proper permutation subgroup.
\end{abstract}

\maketitle

\section{Introduction}\label{sec:intro}

Here is a beautiful and motivating connection between permutations and integers. Let $c_1,c_2,\dots$ be the random cycle lengths of a uniform random permutation from $S_n$, so $\sum_i c_i = n$. Then the normalized lengths $\cP_n = \set{\frac{c_1}{n}, \frac{c_2}{n}, \dots}$ form a random partition of unity with at most $n$ nonzero parts (nonnegative and summing to $1$). A classical result \cite{schmidt1977limit} is that
$$
\cP_n \dconv \Sigma
$$
where $\Sigma$ is the Poisson-Dirichlet distribution, supported on infinite partitions of unity (more on this in Section \ref{sec:part}). Furthermore the distribution function of the largest part $\cM$ of $\Sigma$ is given by the Dickman function $\rho$, that is
\begin{equation}\label{eq:pddickman}
    \pr{\cM \le \frac{1}{x}} = \rho(x)
\end{equation}
where $\rho$ is defined as the unique function satisfying the delay differential equation
 \begin{equation}\label{eq:delaydiff}
     u \rho'(u) + \rho(u-1) = 0
 \end{equation}
 for $u \ge 0$ with initial conditions $\rho(u) = 1$ for $u \in [0,1]$ \cite{tenenbaum2015introduction}.

 Consider now a uniformly random integer $n$ from $[N]$ and factor $n = p_1 p_2 p_3 \dots$ where $p_i$ are primes, not necessarily distinct (repeated with multiplicity). Then another classical result \cite{donngrimm} is that the random partitions of unity $\cQ_N = \set{\frac{\log p_1}{\log n}, \frac{\log p_2}{\log n}, \dots}$ are such that
 $$
\cQ_N \dconv \Sigma
 $$
There is no reason \textit{a priori} why this limit should coincide with that of $\cP_n$. Combining convergence of $\cQ_n$ and $\cP_n$ with \eqref{eq:pddickman} allows for two other characterizations of the Dickman function \cites{tao, moree2014integers}
\begin{equation}\label{eq:dickmanlimitdef}
    \rho(u) = \lim_{x \to \infty} \frac{\abs{\set{n \le x: p_1(n) \le x^{1/u}}}}{x} = \lim_{n \to \infty} \frac{\abs{\set{\sigma \in S_n: c_1(\sigma) \le n/u}}}{n!}
\end{equation}
The fact that these limit functions are the same as the solution to the delay differential equation is again not at all obvious. Surprising results like these have inspired numerous surveys such as \cites{granville2008anatomy, tao} on the anatomy of integers and permutations and even a graphic novel \cite{primesuspects}.

The program of this paper is to shed further light upon these mysterious connections by abstracting away the self-similar structure that underlies random permutations and integers, as initiated in \cite{tao}. In doing so we are able to generalize all the above results. Our generalization of a random permutation will be those with the block structure described in the abstract. The limit of the partitions they define is describable by a simple process of randomly cutting a unit square. It then becomes interesting to explore features of this square cutting distribution $\cP$, as a better understanding of it leads directly to a better understanding of related discrete objects like permutations and integers. 

We focus on the largest part $\cN$ of $\cP$ (analogous to the largest part $\cM$ of $\Sigma$), and show that its distribution function $\pi$ (analogous to $\rho$) can be characterized as the mean of a multiplicative function (analogous to \eqref{eq:dickmanlimitdef}). The distribution function $\pi$ is then numerically approximated to show $\eta \coloneqq \EE \cN \approx 0.40$, meaning that the longest cycle of a random block permutation on $kn$ elements is about $0.40kn$ in expectation for large $k$ and $n$. This should be seen as an analog of the Golomb-Dickman constant $\lambda \coloneqq \EE \cM \approx 0.624$, meaning that the longest cycle of a random permutation on $n$ elements is $0.624n$ in expectation. Finally we generalize further the convergence result and analysis of large parts to the regime where $k$ is fixed, $n$ is large, and the permutations of elements within blocks need not be uniformly sampled. This fills a gap left open by \cite{diaconis2024poisson} where limits of cycle counts and various patterns are derived but their methods are unable to address large cycles. 

\subsection{Partitions of unity and the Poisson-Dirichlet distribution}\label{sec:part}

First we motivate and state the main result, Theorem \ref{thm:sqconv}. Our story starts with the stick breaking interpretation of the Poisson-Dirichlet (PD) distribution on partitions of unity (nonnegative numbers summing to 1, defined more carefully in Section \ref{sec:convpf}), which we will denote $\Sigma$. The PD distribution was generalized to a two-parameter family in \cite{pitman1997two} and continues to appear in recent work on spatial random permutations such as the interchange process \cite{elboim2022infinite}. It is also the equilibrium measure for certain coagulation and fragmentation processes \cite{pitman2006combinatorial}. To construct the PD distribution, consider the following process: break a stick of length 1 uniformly, set the left piece aside, and keep repeating the process on the right piece (see Figure \ref{fig:breakcut} left). We do not recommend trying this at home as you will never finish: almost surely all left pieces after every break has positive width so the process continues indefinitely. The PD distribution is thus supported on partitions with infinitely many parts. The randomness for break points can be sampled as $U_1, U_2, \dots$ iid $U[0,1]$ and the stick lengths are then 
\begin{equation}\label{eq:stickprodstruct}
    \Sigma = \set{U_1, (1-U_1)U_2, (1-U_1)(1-U_2)U_3, \dots}
\end{equation}
We will think of this random partition as being characterized by a self-similarity property
\begin{equation}\label{eq:gemself}
    \Sigma = \set{U,(1-U)\Sigma}
\end{equation}
in distribution, where $U \sim U[0,1]$ independent of $\Sigma$. The square cutting process to follow may be thought of as a two dimensional analog of this stick breaking.

\begin{figure}[h]
    \centering
    \includegraphics[scale=0.45]{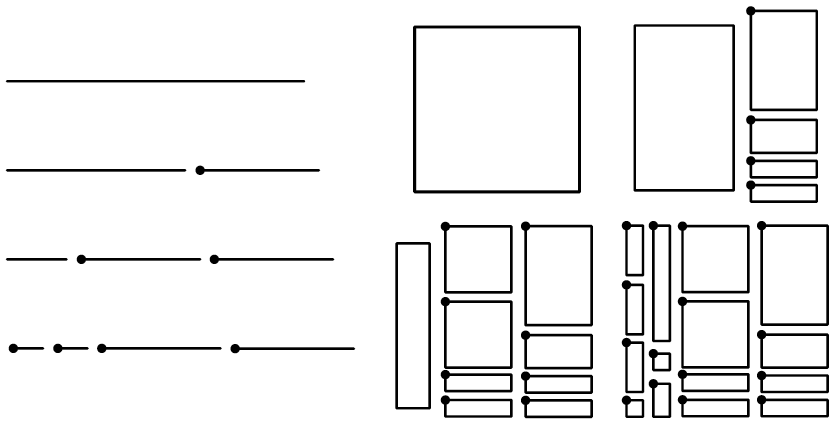}
    \caption{Stick breaking construction of the Poisson-Dirichlet and square cutting construction of $\cP$. Each step of square cutting consists of vertically cutting the leftmost rectangle as well as horizontally cutting the new piece.}
    \label{fig:breakcut}
\end{figure}

\subsubsection{Square Cutting}\label{sec:sqcut}

Consider the following random process that starts with a unit square and ends with infinitely many rectangles. First slice the square starting at a uniform point on the top edge down towards the bottom edge, resulting in two rectangles. Then set aside the right rectangle, and keep repeating on the left (slicing from top to bottom at a uniformly random point). Concurrently have a friend take the set aside rectangles and carry out the same process horizontally on each. That is, your friend slices horizontally from a uniform point on the left edge to the right edge, sets aside the top piece, and repeats on the bottom (see Figure \ref{fig:breakcut} right). The final areas of all rectangles are a random partition of unity from what we call the square cutting distribution and denote $\cP$.

As with stick breaking, the randomness can be sampled by $\set{U_i}_{i=1}^\infty$ to represent the vertical slices and $\set{V_{i,j}}_{i,j=1}^\infty$ for the horizontal, all iid $U[0,1]$ random variables. Encoding the above description the areas of the final pieces are
\begin{align}
    \cP = \{&U_1V_{1,1}; U_1(1-V_{1,1})V_{1,2}; U_1(1-V_{1,1})(1-V_{1,2})V_{1,3}; \dots \label{eq:squareprodstruct}\\
    &(1-U_1)U_2V_{2,1}; (1-U_1)U_2(1-V_{2,1})V_{2,2}; \dots \nonumber \\
    &(1-U_1)(1-U_2)U_3V_{3,1}; \dots \} \nonumber
\end{align}
which exhibits the self-similar structure
\begin{equation}\label{eq:sqself}
    \cP \overset{d}{=} \set{U\Sigma,(1-U)\cP}
\end{equation}
where $U \sim U[0,1]$ and $\Sigma$ is PD distributed, both independent of $\cP$. 

\subsection{Main Results}

Permutations acting on blocks as described in the abstract are known to algebraists as wreath products, which are a certain semidirect product. These terms will only be used at the level of their definitions; no algebra will be needed. Elements $\sigma$ of the wreath product subgroup $S_k^n \rtimes S_n \le S_{kn}$ will be represented as
$$
\sigma = (g_1,\ldots,g_n;h),\qquad g_i\in S_k,\ h\in S_n
$$
This permutes $\{1,\dots,kn\}$ with $g_1$ permuting $\{1,\dots,k\}$, $g_2$ permuting $\{k+1,\dots,2k\}$, $\ldots$, $g_n$ permuting $\{(n-1)k+1,\dots,kn\}$ independently and then $h$ permuting the $n$ blocks. For example $\sigma = ((12),(1)(2),(12);(312)) \in S_2^3 \rtimes S_3$ permutes $1 2 \vert 3 4 \vert 5 6$ first to $2 1 \vert 3 4 \vert 6 5$ and then to $6 5 \vert 2 1 \vert 3 4$. In the usual two-line notation for permutations,
\begin{center}\begin{tabular}{cccccc}
1&2&3&4&5&6\\
6&5&2&1&3&4
\end{tabular}\end{center}
has cycle decomposition $(164)(253)$. Permutation wreath products appear in many areas of mathematics and are often of great interest \cite{diaconis2024poisson}. 

Our main theorem consists of several results that mirror the story told in the previous subsection
\begin{theorem}[Square cutting]\label{thm:sqconveasy}
    Let $\Sigma,\cP$ be partitions of unity defined by
    $$
    \Sigma = (U,(1-U)\Sigma), \quad \cP = (\tilde{U}\Sigma, (1-\tilde{U})\cP)
    $$
    in distribution where $U,\tilde{U}$ are independent and uniformly distributed in the unit interval. Take $\sigma \in S_k^n \rtimes S_n$ uniformly and let $C_1 \ge C_2 \ge \dots$ denote the cycle lengths of $\sigma$. Let $C = \bigp{\frac{C_1}{kn},\frac{C_2}{kn}, \dots}$. Then
    $$
    C \dconv \cP
    $$
    as $k,n \to \infty$. Letting $\cN$ denote the area of the largest piece in $\cP$, $\cN$ has distribution function $\pr{\cN \le u} = \pi(1/u)$ where $\pi$ is defined by
    $$
    \pi(u) \coloneqq \frac{1}{u} (\pi \ast \rho)(u)
    $$
    for $u > 1$ with $\pi(u) = 1$ for $u \in [0,1]$. Consequently the largest cycle behaves in expectation as
    $$
    \lim_{k,n \to \infty} \frac{\EE C_1}{kn} = \eta \coloneqq \EE \cN \approx 0.40
    $$
\end{theorem}
The observant reader may note that $\cP$ is implicitly assumed to have nonincreasing entries, more on this in Section \ref{sec:mainproof}. A closed form expression for $\eta$ is given in Section \ref{sec:large}. To establish a number theoretic connection, recall that the Dickman function $\rho$ can be characterized either through the largest cycle of a random permutation or in terms of smooth integers (integers with small prime factors) \eqref{eq:dickmanlimitdef}. In light of this one may hope to make a similar statement involving $\pi$ which has shown up in the largest cycle of a wreath product permutation. While it would be nice to have such a probabilistic description of $\pi$, it seems the most immediate generalization of this idea lies in thinking of the density of smooth integers as the mean of a multiplicative function. Defining $h_x(p) = 1_{p \le x}$ for $p$ prime and $h_x(nm) = h_x(n)h_x(m)$ for coprime $n,m$, \eqref{eq:dickmanlimitdef} reads
$$
\rho(u) = \lim_{x \to \infty} \frac{1}{x^u} \sum_{n \le x^u} h_x(n)
$$
Here is the analog for $\pi$.
\begin{theorem}[Mean of a multiplicative function]\label{thm:mf}
    Let $f_x(p) = \rho\bigp{\frac{\log p}{\log x}}$ for prime $p$ and extend to $\NN$ totally multiplicatively. Let $\pi$ be as in Theorem \ref{thm:sqconveasy}. Then for $u \ge 0$
    $$
    \pi(u) = \lim_{x \to \infty} \frac{1}{x^u} \sum_{n \le x^u} f_x(n).
    $$
\end{theorem}

The convergence statement in Theorem \ref{thm:sqconveasy} follows directly from a more precise result giving explicit rates. Let $\Delta \subset [0,1]^\NN$ be the space of partitions of unity and $\ell$ a metric on $\Delta$. If $X,Y$ are two random partitions with laws $\mu,\nu$ respectively let
\begin{align*}
    d_\ell(X,Y) &\coloneqq \sup \set{\int_\Delta f(x) d(\mu - \nu)(x): f \text{ continuous and Lip}_\ell (f) \le 1}\\
    &= \inf \EE_{(X,Y)} \ell(x,y)
\end{align*}
be the Wasserstein or transport metric, where the infimum is over all couplings $(X,Y)$.

\begin{theorem}[Rates of convergence]\label{thm:sqconv}
    With notation as in Theorem \ref{thm:sqconveasy}
    $$
    d_{L^\infty}(C,\cP) \le 1/k + 1/n
    $$
    and
    $$
    \frac{\log^2 (kn)}{8kn} \le d_{L^1}(C,\cP) \le \frac{2(\log k + \gamma)}{k} + \frac{2(\log n + \gamma)}{n} + O(1/k^2 + 1/n^2)
    $$
    where $\gamma$ is the Euler-Mascheroni constant and the implied constant is universal.
\end{theorem}

Note that there is no hope of decay in total variation distance between $C$ and $\cP$ as $C$ is supported on rationals almost surely for any $k,n$ but $\cP$ consists almost surely of irrational areas. Thus Wasserstein distance seems a natural choice for these rates. The proof in Section \ref{sec:convpf} actually shows that the $L^\infty$ bound is uniform. That is, under the coupling
$$
\norm{C - \cP}_\infty \le 1/k + 1/n
$$
always. This can be used to bound $d_{L^p}(C,\cP)$ for any $p > 1$ by pulling out the largest term repeatedly until one reaches $d_{L^1}(C,\cP)$.

Using the same coupling we also give a convergence result for the non-normalized cycle counts in Appendix \ref{app} to compound Poisson limits. Without normalization the cycle lengths constitute integer partitions, well-studied objects that in fact gave rise to the Hardy-Littlewood circle method \cite{hardy1918asymptotic}. Compound Poisson limits show up in other places and can be established with an extension of Stein's method \cite{barbour1992compound}, although ours comes from elementary Poisson splitting.

It is worth remarking that the square cutting distribution $\cP$ is actually the PD fragmentation kernel applied to the PD distribution, as defined in \cite{pitman2006combinatorial}*{Section 5}, and falls within the aforementioned two-parameter PD family. The theory of fragmentation processes is much more general, with other combinatorial constructions. For example the same section of \cite{pitman2006combinatorial} shows that the partition defined by the directed graph of a random mapping from $[n]$ to $[n]$ converges to a limit in this two-parameter PD family. In fact the phenomenon that we have been calling self-similarity seems to essentially be the same as one studied extensively in fragmentation processes \cites{bertoin2002self,bertoin2006random}. There is almost certainly further work to be done in this direction. For instance it is known that the cycle partition of a permutation from the Ewens distribution with parameter $\theta$ converges to $\text{PD}(0,\theta)$. It may be the case that the partition from a random wreath permutation where $g_i$ and $h$ are Ewens distributed with parameters $\theta_1,\theta_2$ converges to $\text{PD}(0,\theta_1)$-frag applied to $\text{PD}(0,\theta_2)$ (see Theorem \ref{thm:sqconv}). Even if this is the case, Pitman has communicated to the author that this limit is hard to describe using the typical representations of random partitions.

\subsection{Structure of the paper}\label{sec:struct}
Section \ref{sec:erdtur} extends the Erd\H{o}s-Turán law for the least common multiple of cycle lengths to wreath product subgroups. Section \ref{sec:large} derives the distribution function $\pi$ of the largest part of $\cP$ through self-similarity. It also contains the numerical analysis yielding the constant $\eta \approx 0.40$, touching on connections with the Golomb-Dickman constant. Section \ref{sec:mf} proves Theorem \ref{thm:mf}. With cycle statistics being a strong motivator for studying random partitions like $\pi$ we then extend results to general wreath products of the form $\Gamma^n \rtimes S_n$ which correspond to choosing the permutation of each block uniformly from $\Gamma \subseteq S_k$. These include an extension of the main convergence result to more general random partitions in Section \ref{sec:convpf} (where Theorem \ref{thm:sqconv} is proved) and analysis of the $k$-th largest part for $k \ge 1$ in Section \ref{sec:largeext}.

\subsection{Acknowledgements}

The author thanks Persi Diaconis for encouraging the pursuit of quantitative results and helpful discussions, Jim Pitman for pointing out the references \cites{bertoin2002self,bertoin2006random}, and Kannan Soundararajan for help making the connection to multiplicative functions rigorous.

\section{Erd\H{o}s-Turán law for wreath products}\label{sec:erdtur}

As a warmup we investigate the product structure of cycles in wreath permutations and use it to extend a result of Erd\H{o}s and Turán giving a lognormal distribution for the least common multiple of cycle lengths \cite{erdHos1967some}. Their result has seen generalization to the Ewens sampling distribution \cite{arratia1992limit} but this is the first extension to a uniform permutation from a proper nontrivial subgroup (in fact we do not even require the subset to be a group).

Considering $\sigma = (g_1,\dots,g_n;h) \in S_k^n \rtimes S_n$ its cycle lengths are given by products of cycle lengths involving the $g_i$ and $h$. Consider an $m$-cycle $(1h(1)h^2(1)\dots h^m(1))$ in $h$, where $h^i$ is $h$ composed with itself $i$ times. Then taking $g_1 \circ g_{h(1)} \circ g_{h^2(1)} \circ \dots \circ g_{h^m(1)} \in S_k$ we may decompose it into cycles of length $G_1,\dots,G_l$ and these combine with the $m$-cycle in $h$ to make cycles of lengths $mG_1,\dots,mG_l$ in $\sigma$. This happens for all cycles in $h$ and corresponding compositions of $g$'s. In the example from Section \ref{sec:intro} the only cycle in $h$ is the $3$-cycle $(312)$. Then $g_3 \circ g_1 \circ g_2 = (12) \circ (12) \circ (1)(2) = (1)(2)$ decomposes into two fixed points and the final cycle lengths, two $3$-cycles, are both given as $3 * 1$. Applying this to a simple coupling gives the following. All logarithms throughout the paper are natural.
\begin{theorem}[Wreath product Erd\H{o}s-Turán]\label{thm:erdtur}
    Let $\Gamma \subseteq S_k$ be arbitrary and $C_1,C_2,\dots$ the cycle lengths in a uniform permutation from $\Gamma^n \rtimes S_n$. Let $O_n = \lcm\set{C_1,C_2,\dots}$. Then
    $$
    \frac{\log O_n - \frac{1}{2}\log^2 n}{\sqrt{\frac{1}{3} \log^3 n}} \dconv N(0,1)
    $$
    as $n \to \infty$ for fixed $k$. If $\Gamma = C_k$, the result holds for $k = e^{o(\log^{3/2} n)}$.
\end{theorem}

\begin{proof}
    Let the permutation in the theorem be $\sigma = (g_1,\dots,g_n;h)$ and $P_n$ be the least common multiple of cycle lengths of a uniform $\pi \in S_n$. The original result of Erd\H{o}s and Turán then gives the exact same limit law with $P_n$ in place of $O_n$ \cite{erdHos1967some}. We can then couple $O_n$ and $P_n$ by setting $h = \pi$ so that
    $$
    \log P_n \le \log O_n \le \log \bigp{k!P_n} = \log P_n + O_k(1)
    $$
    which gives the result. The second inequality is because each permutation in $S_k$ of the form $g_1 \circ g_{h(1)} \circ \dots$ decomposes into cycles of length $G_1,G_2 \dots \le k$. These cycle lengths then act as multipliers on the cycle lengths of $h=\pi$ to get the cycle lengths over which $O_n$ is the least common multiple. When multiplying lengths in $h= \pi$ by $G_1,G_2,\dots$, their least common multiple will increase by a factor of at most $\lcm([k]) \le k! = O_k(1)$. If $\Gamma = C_k$ then
    $$
    \log P_n \le \log O_n \le \log P_n + \log k
    $$
    by the above reasoning but since now $G_1,G_2,\dots$ all divide $k$. Then if $\log k = o(\log^{3/2} n)$ the error term converges to $0$ when normalized by $\sqrt{\frac 1 3 \log^3 n}$.
    
\end{proof}

It would be interesting to see which subgroups $\Gamma \le S_k$ allow one to take $k$ large with $n$ in the above theorem, and how large. Let $\lcm(\sigma)$ be the least common multiples of cycle lengths of a permutation $\sigma$. Then for instance with $\Gamma = S_k$ one has
$$
O_n = \lcm(C_1 * \lcm(\pi_1), C_2 * \lcm(\pi_2), \dots)
$$
where $\pi_1, \pi_2, \dots$ are iid uniform from $S_k$ and $C_1,C_2,\dots$ are the cycle lengths in a uniform permutation from $S_n$. It is not clear how large $k$ can be taken in this case.

\section{Largest part and integral equations}\label{sec:large}

This section derives the distribution of the largest part of the square cutting distribution $\cP$ as defined in Theorem \ref{thm:sqconveasy}. First recall from Section \ref{sec:part} that the Dickman function $\rho$ describes the distribution of the largest part of the PD distribution $\Sigma$. Thus this result is essentially determining the analog of $\rho$ for $\cP$. 

\begin{theorem}[Largest piece from square cutting]\label{thm:sqlarge}
    Let $\cN$ be the area of the largest piece in $\cP$. Then $\cN$ has distribution function $\pr{\cN \le u} = \pi(1/u)$ where $\pi$ is defined by
    $$
    \pi(u) \coloneqq \frac{1}{u} (\pi \ast \rho)(u)
    $$
    for $u > 0$ with $\pi(u) = 1$ for $u \in [0,1]$.
\end{theorem}
\begin{proof}
Let $\cM$ be the largest element of $\Sigma$ (length of the longest stick). Then
\begin{align*}
    \phi(x) &\coloneqq \pr{\cN \le 1/x} = \pr{\max(U\cM,(1-U)\cN) \le 1/x} \\
    &= \int_0^1 \pr{\max(t\cM,(1-t)\cN) \le 1/x} dt\\
    &= \int_0^1 \pr{\cM \le \frac{1}{xt}}\pr{\cN \le \frac{1}{x(1-t)}} dt\\
    &= \int_0^1 \rho(xt) \phi(x(1-t)) dt\\
    &= \frac{1}{x} \int_0^x \rho(y) \phi(x-y) dy
\end{align*}
which is exactly the defining property of $\pi$, so $\phi=\pi$. The penultimate equality is \eqref{eq:pddickman}, which can in fact be proved by running this exact proof and noting that
$$
\rho(u) = \frac{1}{u}\int_{u-1}^u \rho(t) dt = \frac{1}{u}\int_{0}^u \rho(t) 1_{[0,1]}(u-t) dt = \frac{1}{u}(\rho \ast 1_{[0,1]})(u).
$$
One may then use \eqref{eq:gemself} instead of \eqref{eq:sqself} and the exact same steps to verify that $\pr{\cM \le 1/x}$ satisfies this convolution equation.
\end{proof} The proof illustrates how self-similarity properties like \eqref{eq:gemself} and \eqref{eq:sqself} naturally lead to a convolutional structure on the distribution of large parts of partitions. It is then of interest to study general distribution functions of this convolutional form. While a notion of self-similarity is mentioned in \cite{tao} and the definition of $\rho$ as the solution to a delay differential equation \eqref{eq:delaydiff} can be seen in this way, the author is not aware of a direct connection through convolution. Viewing $\rho$ as defined by
$$
\rho(u) = \frac{1}{u}(\rho \ast 1_{[0,1]})(u)
$$
places it in a family of solutions to what is called integral equations. For any function $\psi: [0,\infty) \to [0,1]$ with $\psi(x) = 1$ for $x \in [0,1]$, consider the function $\phi$ defined by
\begin{equation}\label{eq:def}
    \phi(x) = \frac{1}{x}(\phi \ast \psi)(x) = \frac{1}{x}\int_{0}^x \phi(t)\psi(x-t)dt
\end{equation} 
with $\phi(x) = 1$ for $x \in [0,1]$. Such equations arise in the study of the spectrum of multiplicative functions in \cite{granville2001spectrum}, where basic existence and uniqueness properties are established. These also show up as the stationary distributions of iterated random functions as in \cite{diaconis1999iterated}. Until now they are absent in the literature on the Dickman function or connections between permutations and integers. Dickman's function $\rho$ and $\pi$ above are examples, and $\pi$ can be thought of as one step above $\rho$ in a self-similarity hierarchy: $\pi$ is defined via convolution with $\rho$ which is in turn defined via convolution with $1_{[0,1]}$. 

It would be interesting to study distributions with density proportional to the solution of an integral equation (as opposed to distribution function as we have seen so far). This has been done in the case of the Dickman function and there are many elegant and fundamental constructions of such a distribution, some even relating to cycles of permutations, as well as established properties like infinite divisibility \cite{molchanov2020dickman}. Essentially from definition, distributions with density proportional to the solution of an integral equation have characteristic (random) operators that allow for Stein's method. If $X$ has density $\phi$ and $Y$ has density $\psi$ related as in (\ref{eq:def}) then
$$
\EE[Xf(X)] = \int_0^\infty tf(t)\phi(t)dt = \int_0^\infty f(t) (\phi \ast \psi)(t) dt = \EE[f(X+Y)]
$$
Applying this to the laplace transform $L_X(s) = \EE e^{-sX}$ gives the self-similarity property
\begin{equation*}\label{eq:lap}
    L'_X(s) = -L_Y(s)L_X(s)
\end{equation*}
which can be used to derive the transform of random variables with Dickman or $\pi$ density. This is essentially the fact that the Fourier transform takes convolution to multiplication.

As with the Dickman function $\rho$, saying anything concrete about $\pi$ from its recursive definition is a challenge. As a first stab at numerical analysis we use a slightly shifted Euler method with step size $0.001$ to solve the delay differential equation \eqref{eq:delaydiff} and approximate values of $\rho(t)$ for $1 \le t \le 200$ (with $\rho(t) = 1$ for $t \in [0,1]$). These values are then used to approximate values of $\pi(t)$ also for $1 \le t \le 200$, using the composite trapezoidal rule in a discrete approximation of the convolution integral. Figure \ref{fig:plot} plots the approximations and Table \ref{tab:data} provides some explicit values.

\newlength{\graphicwidth}
\newlength{\tablewidth}
\newlength{\myheight}
\setlength{\graphicwidth}{0.55\textwidth}  
\setlength{\tablewidth}{0.45\textwidth}    
\setlength{\myheight}{3cm}                

\begin{table}[h]
\begin{minipage}[c]{.6\textwidth}
    \centering
    \begin{tabularx}{\linewidth}{|c|X|X|X|}
      \hline
      \rowcolor{gray!30}
      $u$ & $\rho(u)$ & $\tilde{\rho}(u)$ & $\tilde{\pi}(u)$ \\ \hline
      1    & $1$  & $1$    & $1$   \\ \hline
      2    & $0.3069$    & $0.3071$    & $0.7594$    \\ \hline
      3    & $0.04861$    & $0.04871$    & $0.4247$    \\ \hline
      4    & $0.004911$  & $0.004930$    & $0.1969$   \\ \hline
      5    & $0.0003547$ & $0.0003567$  & $0.07996$ \\ \hline
      6    & $1.965 \times 10^{-5}$  & $1.980 \times 10^{-5}$  & $0.02932$ \\ \hline
      7    & $8.746 \times 10^{-7}$ & $8.830 \times 10^{-7}$  & $0.009904$ \\ \hline
      8    & $3.232 \times 10^{-8}$ & $3.271 \times 10^{-8}$ & $0.003125$ \\ \hline
      9    & $1.016 \times 10^{-9}$ & $1.031 \times 10^{-9}$  & $0.0009304$ \\ \hline
      10   & $2.77 \times 10^{-11}$    & $2.816 \times 10^{-11}$    & $0.0002634$    \\ \hline
    \end{tabularx}

    \vspace{1ex} 
    
    \begin{tabularx}{\linewidth}{|X|X|X|}
      \hline
      \rowcolor{gray!30}
       $\lambda$ & $\tilde{\lambda}$ & $\tilde{\eta}$ \\ \hline
       $0.62433000$    & $0.62437140$    & $0.39688143$ \\ \hline
    \end{tabularx}
  \end{minipage}
  \caption{Function and constant approximations, where $\tilde{\cdot}$ denotes our approximated values of Dickman's function $\rho$, the Golomb-Dickman constant $\lambda$, and $\pi, \eta$ the analogous function and constant related to the square-cutting distribution $\cP$. True (better approximated) values of $\rho$ are taken from Wikipedia.} \label{tab:data}
\end{table}

\begin{figure}[h]
    \centering
    \includegraphics[width=.75\textwidth]{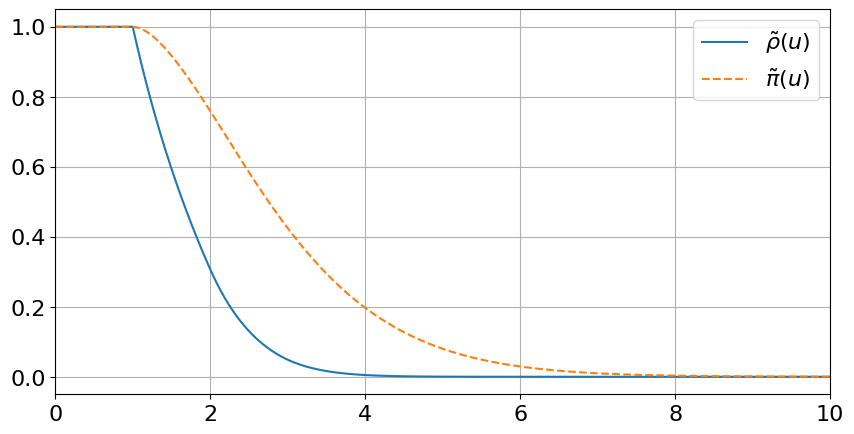}
    \caption{Function approximations for values $0 \le u \le 10$.}
    \label{fig:plot}
\end{figure}


    

Using these values of $\pi$ we then approximate the expectation of the largest part $\cN$ of the square cutting distribution. Explicitly we use the layer cake formula for expectation
$$
\eta \coloneqq \EE \cN = \int_0^\infty 1-\pr{\cN \le t} dt = \int_0^1 1-\pi(1/t) dt = \int_1^\infty \frac{1-\pi(t)}{t^2} dt
$$
and another discrete approximation with the trapezoidal rule is used to approximate
$$
\int_1^{200} \frac{1-\pi(t)}{t^2} dt \approx 0.39188143
$$
Since $\pi(200) \approx 5.99 \times 10^{-20}$ the remaining tail of the integral is around $\int_{200}^\infty t^{-2} dt = 1/200$, and adding this in gives $\EE \cN \approx 0.39688143$. Taking into account the fact that the approximated values for $\rho$ seem to be slightly biased upwards compared to the true values one may also expect a positive bias in the approximations of $\pi$ (though this is not rigorous). This could translate to a slight negative bias in the integrand of $\EE \cN$ and an underestimate of this constant. Overall given the various sources of inaccuracy we simply report $0.40$ and hope that the error is not so severe as to really make the constant outside $[0.395, 0.405)$. The same approximation of $\rho$ and same trapezoidal rule approximate the Golomb-Dickman constant to high accuracy (see Table \ref{tab:data}), which provides some confidence.

In contrast the expected length of the largest part in the PD distribution $\Sigma$ is known as the Golomb-Dickman constant $\lambda$ and is around $0.624$. If $a_n$ is the largest cycle of a random permutation from $S_n$ and $p_1(n)$ the largest prime factor of the integer $n$, then on account of the classical convergence results for cycles and prime factors to $\Sigma$ in Section \ref{sec:part} this constant is both 
$$
\lambda = \lim_{n \to \infty} \frac{\EE a_n}{n}
$$
and
$$
\lambda = \lim_{N \to \infty} \EE_{n \in [N]} \frac{\log p_1(n)}{\log n}
$$
There is however yet another surprising number theoretic characterization
$$
\lambda = \lim_{N \to \infty} \PP_{n \in [N]}\bigp{p_2(n) \le \sqrt{p_1(n)}}
$$
where $p_2(n)$ is the second largest prime factor and the probability is over uniform $n \in [N]$. Theorem \ref{thm:sqconv} implies that if $\eta \approx 0.40$ is the constant above and $b_{k,n}$ is the largest cycle of a random permutation from $S_k^n \rtimes S_n$ then
$$
\eta = \lim_{k,n \to \infty} \frac{\EE b_{k,n}}{kn}
$$
It would be interesting to find analogs of the number theoretic characterizations for $\eta$. While we do not know how to do this, we establish a number theoretic characterization of $\pi$ in the next section.

\section{A multiplicative function (proof of theorem \ref{thm:mf})}\label{sec:mf}

We establish a way to approximate a sum over primes by an integral. First we need the de la Vallée Poussin approximation to the prime counting function.

\begin{lemma}[\cite{primedisting}*{Theorem 23}]\label{lem:primapprox}
    Let $\tau(x) \coloneqq \abs{\set{p: p \le x}}$ be the prime counting function. Then
    $$
    \tau(x) = \int_{2}^x \frac{dt}{\log t} + O(xe^{-a\sqrt{\log x}})
    $$
    where $a$ is an absolute positive constant.
\end{lemma}

\begin{lemma}[Integral approximation]\label{lem:intapprox}
    For large $x$ and an index function $c: \RR_+ \to \RR_+$, let $g_{c(x)}: [1,\infty] \to [0,1]$ be a family of continuous functions where $g_{c(x)}$ is differentiable everywhere except possibly at $c(x)$. Assume further that $\lim_{y \to \infty} g_{c(x)}(y) = 0$, $\int_{c(x)}^\infty g_{c(x)}(t) dt < \infty$, $g_{c(x)}(x) = o\bigp{e^{\eps \sqrt{\log x}}}$, and $g_{c(x)}(2) = o\bigp{e^{\eps \sqrt{\log x}}}$ for every fixed positive $\eps$. Then 
    $$
    \sum_{p \le x} g_{c(x)}(p) = \int_2^{x} \frac{g_{c(x)}(t)}{\log t} dt + o(x)
    $$
\end{lemma}
\begin{proof}
In the event $g_{c(x)}$ is not differentiable at $c(x)$ (but still continuous there), for $y < c(x)$ we will use the improper integral
$$
\int_y^\infty g'_{c(x)}(t) dt \coloneqq \bigp{\int_y^{c(x)} + \int_{c(x)}^\infty} g'_{c(x)}(t) dt = -g_{c(x)}(y)
$$
to simplify notation. Using summation by parts and applying Lemma \ref{lem:primapprox} twice
    \begin{align*}
        \sum_{p \le x} g_{c(x)}(p) &= - \sum_{p \le x} \int_2^\infty g'_{c(x)}(t) 1_{t \ge p} dt= -\int_2^\infty g'_{c(x)}(t) \tau(x \wedge t) dt\\
        &= -\int_2^x g'_{c(x)}(t) \tau(t) dt - \tau(x) \int_x^\infty g'_{c(x)}(t) dt\\
        &= \tau(x) g_{c(x)}(x) - \int_2^x g'_{c(x)}(t) \tau(t) dt \\
        &= g_{c(x)}(x)\int_2^x \frac{1}{\log t} dt + O\bigp{g_{c(x)}(x) xe^{-a\sqrt{\log x}}} \\
        &- \int_2^x g'_{c(x)}(t) \int_2^t \frac{1}{\log s} ds dt + O\bigp{\int_2^x g'_{c(x)}(t)te^{-a\sqrt{\log t}} dt} \\
    \end{align*}
    for some positive $a$. Integrating by parts again the first term on the last line
    $$
    \int_2^x g'_{c(x)}(t) \bigp{\int_2^t \frac{1}{\log s} ds} dt = g_{c(x)}(x)\bigp{\int_2^x \frac{1}{\log s} ds} - \int_2^x \frac{g_{c(x)}(t)}{\log t} dt
    $$
    and bounding the error term
    $$
    \int_2^x g'_{c(x)}(t)te^{-a\sqrt{\log t}} dt \le xe^{-a\sqrt{\log x}} \int_2^x g'_{c(x)}(t) dt = g_{c(x)}(x)xe^{-a\sqrt{\log x}} - g_{c(x)}(2)xe^{-a\sqrt{\log x}}.
    $$
    Thus by assumption that $g_{c(x)}(x), g_{c(x)}(2)$ are small, all three error terms combine to be $o(x)$ and
    \begin{align*}
        \sum_{p \le x} g_{c(x)}(p) &= g_{c(x)}(x)\int_2^x \frac{1}{\log t} dt - \bigp{g_{c(x)}(x)\bigp{\int_2^x \frac{1}{\log s} ds} - \int_2^x \frac{g_{c(x)}(t)}{\log t} dt} + o(x) \\
        &= \int_2^x \frac{g_{c(x)}(t)}{\log t} dt + o(x)
    \end{align*}
\end{proof}

The last ingredients to prove Theorem \ref{thm:mf} are results of Granville and Soundararajan on solutions to integral equations. Considering any $\chi: [0,\infty) \to [0,1]$, \cite{granville2001spectrum}*{Theorem 3.3} asserts that there is a unique solution  $\sigma: [0, \infty) \to [0,1]$ such that
$$
u\sigma(u) = (\sigma \ast \chi)(u).
$$

\begin{proposition}[\cite{granville2001spectrum}*{Proposition 1}]\label{prop:spectrum}

Let $\sigma,\chi$ be as above and $g_x: \NN \to [0,1]$ a multiplicative function such that $g_x(n) = 1$ for $n \le x$. Let $\cV(x) = \sum_{p \le x} \log p$ and define
$$
\chi(u) = \chi_{g_x}(u) = \frac{1}{\cV(x^u)}\sum_{p \le x^u} g_x(p) \log p
$$
Then $\chi(t)$ takes values in $[0,1]$ and $\chi(t) = 1$ for $t \in [0,1]$. Furthermore
$$
\frac{1}{x^u} \sum_{p \le x^u} g_x(p) = \sigma(u) + O\bigp{\frac{u}{\log x}}.
$$
\end{proposition}

\begin{proof}[Proof of Theorem \ref{thm:mf}]
    First note that the equivalence trivially holds for $u = 0$ by the initial conditions $\pi(0) = \rho(0) = 1$. It remains to verify the equality for $u > 0$. Taking $\chi = \rho$ we have that $\sigma = \pi$ is unique and well defined. Note that if we set $g_x(p) = \rho \bigp{\frac{\log p}{\log x}}$ multiplicative then $g_x(n) = 1$ for $n \le x$ by the initial conditions of $\rho$, and similarly $g_x$ takes values in $[0,1]$. Thus by Proposition \ref{prop:spectrum} it suffices to verify that
    $$
    \rho(u) = \lim_{x \to \infty} \frac{1}{\cV(x^u)} \sum_{p \le x^u} \rho\bigp{\frac{\log p}{\log x}} \log p 
    $$
    as the limit may then be passed by bounded convergence inside the convolution. It is a classical result that the Chebyshev summatory function has asymptotics
    $$
    \cV(x) = x + o(x)
    $$
    (more precise bounds are known but not needed for our purposes, see Wikipedia for further references). Now we check that, for $x > 1$, $g_x(y) = \rho\bigp{\frac{\log y}{\log x}} \log y$ satisfies the assumptions of Lemma \ref{lem:intapprox}. Indeed the Dickman function is continuous and differentiable everywhere except at $1$. A simple upper bound for $\rho$ is $\rho(x) \le \frac{1}{\floor{x}!}$, so we also have that for any fixed $x$ and sufficiently large $y$
    $$
    g_x(y) \le \bigp{\floor{\frac{\log y}{\log x}}!}^{-1} \log y \to 0.
    $$ 
    Finally if we fix any $u > 0$ and set the index function $c(x) = x^{1/u}$, then $g_x(x^u) = u\rho(u)\log x =o\bigp{e^{\eps \sqrt{u \log x}}}$ for every positive $\eps$ and $g_x(2) \le 1$ for sufficiently large $x$. Applying Lemma \ref{lem:intapprox} then gives
    \begin{align*}
        \sum_{p \le x^u} \rho\bigp{\frac{\log p}{\log x}} \log p &= \int_2^{x^u} \rho\bigp{\frac{\log t}{\log x}} dt + o(x^u)
    \end{align*}
    so we are left with (using an improper integral as in the proof of Lemma \ref{lem:intapprox})
    \begin{align*}
        &\lim_{x \to \infty} \frac{\int_2^{x^u} \rho\bigp{\frac{\log t}{\log x}} dt + o(x^u)}{x^u + o(x^u)} = \lim_{x \to \infty} \frac{1}{x^u} \int_2^{x^u} \rho\bigp{\frac{\log t}{\log x}} dt \\
    &= \lim_{x \to \infty} \frac{1}{ux^{u-1}} \bigp{ux^{u-1}\rho(u) - \int_2^{x^u} \rho'\bigp{\frac{\log t}{\log x}}\frac{\log t}{x \log^2 x} dt}
    \end{align*}
   to show the limit is $\rho(u)$ we bound the improper integral. Noting that $\rho'$ is nonpositive where it exists
   \begin{align*}
       \int_2^{x^u} -\rho'\bigp{\frac{\log t}{\log x}}\frac{\log t}{x \log^2 x} dt &\le \frac{-u}{x \log x} \int_2^{x^u} \rho'\bigp{\frac{\log t}{\log x}} dt \\
       &= \frac{u}{x \log x} \bigp{\rho \bigp{\frac{\log 2}{\log x}} - \rho \bigp{u}} \\
       &\le \frac{u}{x \log x} \to 0.
   \end{align*}
\end{proof}

\section{Rates of convergence (proof of theorem \ref{thm:sqconv}) and extension}\label{sec:convpf}

\subsection{Coupling}\label{sec:couple}

This section provides the key coupling methods used in the proofs of Theorems \ref{thm:sqconv} and \ref{thm:nonnormlim}. We start with briefly recalling the Feller coupling, which gives a way to directly sample the cycle type of a permutation using a binary sequence without dealing at all with the permutation itself.

Let $\xi_1, \xi_2, \dots$ be independent such that $\pr{\xi_i = 1} = 1/i$ and $\xi_i = 0$ otherwise. We call the pattern $\xi_i \xi_{i+1} \dots \xi_{i+k} = 10^{k-1}1$ a $k$-space starting at $i$. Let $A = (A_1,A_2,\dots)$ and $Z = (Z_1,Z_2,\dots)$ be such that $A_k$ denotes the number of $k$-spaces in $\Xi_n \coloneqq 1 \xi_2 \xi_3 \dots \xi_n 1$ and $Z_k$ denotes the number of $k$-spaces in $\Xi \coloneqq 1 \xi_2 \xi_3 \dots$. Then the Feller coupling asserts that the number of $k$-cycles in a uniform permutation from $S_n$ is distributed as $A_k$ \cites{najnudel2020feller, diaconis2024poisson}, $A_k \to Z_k$ as $n \to \infty$, and that $Z_k \sim \Poi(1/k)$ independently \cite{ignatov1982constant}.

We will also use a normalized version of the Feller coupling, that allows one to obtain the normalized cycle lengths from a Poisson process on $[0,1]$. Looking at the binary sequence $\Xi_n$ on the unit interval, with $\xi_i$ marking $(i-1)/n$, the occurence of $1$'s may be thought of as a process where the lengths of time between arrivals gives a partition of unity arising from the permutation. That is
$$
\cA_n \coloneqq \bigp{\frac{n}{n}, \dots, \frac{n}{n}, \frac{n-1}{n},\dots,\frac{n-1}{n},\dots, \frac k n, \dots, \frac k n, \dots, \frac 1 n, \dots, \frac 1 n}
$$
is a partition of unity arising from $A$ (and hence $\Xi_n$) where parts of size $\frac k n$ are repeated with multiplicity $A_k$. The Feller coupling gives that this partition is equal in distribution to one defined by normalizing cycle lengths of a uniform permutation. We go one step further and find a partition equal in distribution to $\cA_n$ that we can couple to the Poisson process $X$ on the interval $t \in [0,1]$ with rate $1/t$. This Poisson process can intuitively be thought of as the limit $Z_k$ ``rescaled'' to the unit interval.

Let $x_1 \ge x_2 \ge \dots \in [0,1]$ be the arrival times of $X$ and $X[a,b] = \abs{\set{x_i \in [a,b]}}$. Note that
$$
\pr{X\bigb{\frac{j-1}{n}, \frac{j}{n}} = 0} = \exp \bigp{-\int_{\frac{j-1}{n}}^{\frac{j}{n}} \frac{1}{t} dt} = 1- \frac{1}{j}
$$
which is exactly $\pr{\xi_j} = 0$. It is known \cite{ignatov1982constant} that the partition $\Sigma \coloneqq (1-x_1, x_1-x_2,x_2-x_3,\dots)$ is Poisson-Dirichlet distributed, and indeed the above computation can show that $1-x_1$ is uniformly distributed in the interval, linking this to the stick breaking interpretation in Section \ref{sec:intro}; this coupling can equivalently be carried out by coupling uniform with discrete uniform. To couple round $y_i = \frac
{\floor{nx_i}}{n}$ and let $\Sigma_n \coloneqq (1-y_1,y_1-y_2,\dots)$ with at most $n$ nonzero entries. From the above computation $\Sigma_n \overset{d}{=} \cA_n$ and thus is also equal in distribution to the random partition arising from the normalized cycle lengths of a uniform $\sigma \in S_n$. In conclusion, rounding points of $X$ gives a coupling between the Poisson-Dirichlet $\Sigma$ and the cycle partition $\Sigma_n$ of a uniform $\sigma \in S_n$.

\subsection{Proof of Theorem \ref{thm:sqconv}}\label{sec:mainproof}

This section first proves Theorem \ref{thm:sqconv} and then generalizes to $\Gamma^n \rtimes S_n$ for arbitrary $\Gamma \le S_k$ with only $n \to \infty$. A similar generalization is easily achieved in the same way for $S_k^n \rtimes \Gamma$ for $\Gamma \le S_n$ and $k \to \infty$ but we do not write out the details here. It should be noted that if one desires only a weak convergence statement of Theorem \ref{thm:sqconv} without rates then coupling is not needed; the analogous statement for uniform random permutations from $S_k$ and $S_n$ can be bootstrapped to give convergence for $S_k^n \rtimes S_n$. 

We must formalize the space of partitions of unity introduced in Section \ref{sec:part}. One generally hopes to define partitions to be invariant under permutation, but doing so makes them hard to compare and operate with. We follow the lead of past literature and sort them in nonincreasing order, stating results for ordered partitions as vectors. Define the infinite simplex and its subset of nonincreasing vectors
\begin{align*}
    \bar{\Delta} &\coloneqq \set{x = (x_1,x_2,\dots): x_i \ge 0, \text{ and } \sum_{i=1}^\infty x_i = 1}, \quad \Delta \coloneqq \set{x \in \bar{\Delta}: x_1 \ge x_2 \ge \dots}
\end{align*}
There is then a natural projection from $\bar{\Delta} \to \Delta$ that simply sorts a vector. We will think of a partition of unity as being both the preimage of an $x \in \Delta$ (like a multiset) as well as $x$ itself, depending on what is more convenient. Thus if it is asserted that $x = y$ for two partitions that are not necessarily nonincreasing, what is really meant is that their projections to $\Delta$ are equal.

Some key features of the coupling $(\Sigma_n,\Sigma)$ in section \ref{sec:couple} are
\begin{equation}\label{eq:coupprop}
    \norm{\Sigma - \Sigma_n}_\infty \le 1/n, \quad \EE\norm{\Sigma - \Sigma_n}_1 \le \frac{3(\log n + \gamma)}{n} + O(1/n^2)
\end{equation}
Both follow from definition of $\Sigma_n$ as distances between rounded points $y_i$, whereas $\Sigma$ is the distances between unrounded points $x_i$ (define $x_0 = y_0 = 1$). Recall that 
\begin{equation}\label{eq:rounding}
    0 \le x_i - y_i = x_i - \frac{\floor{n x_i}}{n} \le \frac{1}{n}.
\end{equation}
The first part of \eqref{eq:coupprop} is immediate, so consider the second. For a given $i \in [n]$ define
$$
I(i) \coloneqq \left[\frac{(i-1)}{n}, \frac{i}{n}\right), \quad X(i) \coloneqq \set{k: x_k \in I(i)}.
$$
Note that $X(i)$ may be empty and that if it is not among $k \in X(i)$ we have $y_{k-1}-y_k = 0$ unless $k = \min X(i)$, ie $x_k$ is the rightmost point in $I(i)$. This is because for all $k \in X(i)$ not the minimum $x_{k},x_{k-1} \in I(i)$ so they both round to the same point and $y_{k-1} = y_{k}$ (see Figure \ref{fig:coup} but only consider points on the top edge). Then the number of these rightmost points $x_k$ is distributed exactly as the number of nonzero parts $y_{k-1}-y_k$ of $\Sigma_n$, which we denote $\abs{\Sigma_n}$. We now bound the contribution to $\norm{\Sigma - \Sigma_n}_1$ from parts $x_{k-1}-x_k$ with $k \in X(i)$. The total length of such parts where $k \neq \min X(i)$ is at most $\frac 1 n$, and for $k=\min X(i)$ by \eqref{eq:rounding} we have
$$
\abs{(x_{k-1}-x_k) - (y_{k-1}-y_k)} \le \frac 1 n.
$$
Summing over the contribution from $I(i)$ for each $i$ gives
$$
\norm{\Sigma_n - \Sigma}_1 \le \frac{2}{n} \abs{\Sigma_n}.
$$
Using that the expected number of cycles in a uniform permutation from $S_n$ is the harmonic number $H_n$ we have
$$
\EE\norm{\Sigma_n - \Sigma}_1 \le \frac 2 n \EE\abs{\Sigma_n} = \frac{2}{n} H_n \le \frac{2(\log n + \gamma)}{n} + O(1/n^2).
$$

In order to extend this coupling to wreath products we need any pairing function $\NN^2 \to \NN$ to collapse a sequence of partitions into a single partition. To make an explicit choice take $g(i,j) = 2^{i-1}(2^j-1)$ for $i,j \in \NN$. Then define $f: \Delta^\NN \to \Delta$ given by
$$
    f(a,b_1,b_2,\dots)_{g(i,j)} = a_i \cdot (b_i)_j
$$
We may now couple $C$ and $\cP$. Write $\sigma = (g_1,g_2,\dots,g_n;h)$ as in Section \ref{sec:intro} and let 
\\
$a_1,a_2,\dots,a_m$ be the cycles of $h$. For $i \in [m]$, with $a_i = (e,h(e),h^2(e),\dots,h^{\abs{a_i}}(e))$, let $(b_i)_1,(b_i)_2,\dots$ be the cycle lengths of $g_{e} \circ g_{h(e)} \circ g_{h^2(e)} \circ \dots \circ g_{h^{\abs{a_i}}(e)}$. Then the product structure established in Section \ref{sec:erdtur} gives that $C$ is the same as $f(a,b_1,b_2,\dots,b_m)$ if $a = \bigp{\frac{\abs{a_1}}{n}, \frac{\abs{a_2}}{n},\dots,\frac{\abs{a_m}}{n}}$ and $b_i = \bigp{\frac{\abs{(b_i)_1}}{k}, \frac{\abs{(b_i)_2}}{k},\dots}$.

It thus suffices to bound the distance between $f(a,b_1,b_2,\dots)$ and $\cP$, where $a \sim \Sigma_n$ and $b_i \sim \Sigma_k$ iid. Note also that $\cP = f(\Sigma,\Sigma,\dots)$ by \eqref{eq:stickprodstruct} and \eqref{eq:squareprodstruct}. Applying the above coupling $(\Sigma_n,\Sigma)$ for $a$ (distributed $\Sigma_n$) and each of $b_1,b_2,\dots$ (distributed iid $\Sigma_k$) yields a grand coupling $(a,\Sigma_a), (b_1,\Sigma_{b_1}),(b_2, \Sigma_{b_2}),\dots$ where $\Sigma_a, \Sigma_{b_i}$ are iid $\Sigma$ (Figure \ref{fig:coup}). By \eqref{eq:coupprop} $\norm{a - \Sigma_a}_\infty \le 1/n$ and $\norm{b_i - \Sigma_{b_i}}_\infty \le 1/k$ for every $i$. It follows that
$$
\norm{f(a,b_1,b_2,\dots) - f(\Sigma_a,\Sigma_{b_1},\Sigma_{b_2},\dots)}_\infty \le \frac{1}{k} + \frac{1}{n}
$$
Now considering $i$ arbitrary we have
$$
\norm{a_i b_i - (\Sigma_a)_i \Sigma_{b_i}}_1 \le \abs{a_i - (\Sigma_a)_i} + a_i \norm{b_i - \Sigma_{b_i}}_1
$$
Summing this over all $i$ gives
$$
\norm{f(a,b_1,b_2,\dots) - f(\Sigma_a,\Sigma_{b_1},\Sigma_{b_2},\dots)}_1 \le \norm{a - \Sigma_a}_1 + \sum_i a_i \norm{b_i - \Sigma_{b_i}}_1
$$
and finally taking expectations and using that $b_i$ are iid and independent of $a$
\begin{align*}
    d_{L^1}(C,\cP) &= \EE\norm{f(a,b_1,b_2,\dots) - f(\Sigma_a,\Sigma_{b_1},\Sigma_{b_2},\dots)}_1 \\
    &\le \EE\norm{a - \Sigma_a}_1 + \sum_i \EE a_i \EE \norm{b_i - \Sigma_{b_i}}_1 \\
    &= \EE\norm{a - \Sigma_a}_1 +  \EE \norm{b_1 - \Sigma_{b_1}}_1 \EE \sum_i a_i \\
    &\le \frac{2}{n}H_n + \frac{2}{k}H_k
\end{align*}

\begin{figure}[h]
    \centering
    \includegraphics[scale=0.6]{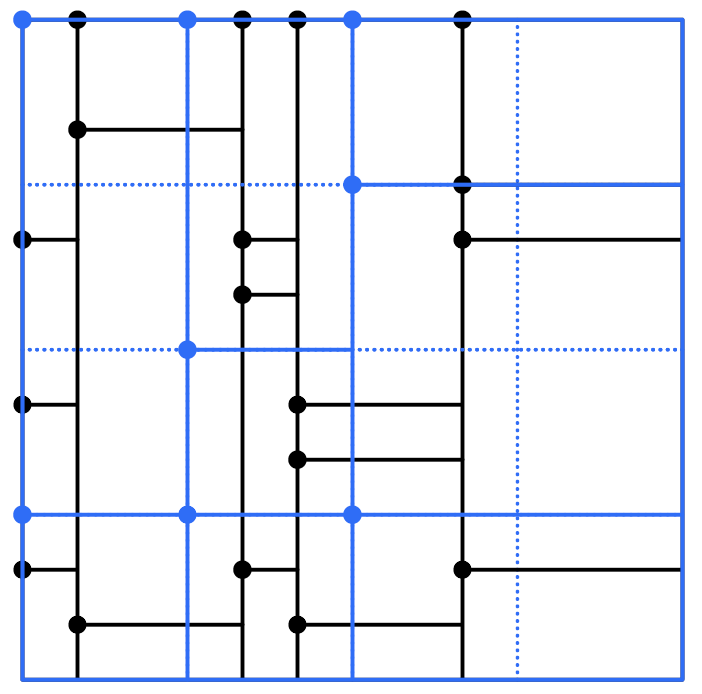}
    \caption{Coupling the square partition defined by the continuous time Poisson process $\Sigma$ (black) with the partition $\Sigma_n$, with break points at multiples of $1/n$, for $n=4$. In reality $\Sigma$ has infinitely many parts a.s. Given a blue rectangle in $\Sigma_n$ with top left corner $y$ we compare it with the black rectangle in $\Sigma$ with top left corner $x$ where $x$ is the rightmost and then bottom most point across from $y$ in the dashed blue square with top left corner $y$.}
    \label{fig:coup}
\end{figure}

Finally for the lower bound note that
$$
nk\norm{a_i b_i - (\Sigma_a)_i \Sigma_{b_i}}_1 = \sum_{j=1}^\infty \abs{na_i \cdot k (b_i)_j - n(\Sigma_a)_i \cdot k(\Sigma_{b_i})_j} \ge \sum_{j=1}^\infty \norm{n(\Sigma_a)_i \cdot k(\Sigma_{b_i})_j}_\ZZ
$$
where $\norm{\cdot}_\ZZ$ denotes distance to the nearest integer, since $na_i \cdot k(b_i)_j \in \ZZ$. To handle this term we use a result of Kingman \cite{kingman1975random} that
$$
\EE_\Sigma \sum_{i=1}^\infty \phi(\Sigma_i) = \int_0^1 \phi(t) \frac{1}{t} dt
$$
for any nonnegative $\phi$ that makes the right hand side convergent, which can be seen from the Poisson process construction of the Poisson-Dirichlet \cites{arratia2006tale, ignatov1982constant}. Then
\begin{align*}
    \psi(x) &\coloneqq \EE_\Sigma \sum_{i=1}^\infty \norm{x \cdot k\Sigma_i}_\ZZ = \int_0^1 \norm{x \cdot ky}_\ZZ \frac{1}{y} dy \ge \sum_{m=1}^{\floor{kx-1/2}} \int_{m-1/2}^{m+1/2} \frac{\norm{z}_\ZZ}{z} dz \\
    &= \sum_{m=1}^{\floor{kx-1/2}} \int_{m-1/2}^{m} \frac{m-z}{z} dz + \int_{m}^{m+1/2} \frac{z-m}{z} dz = \sum_{m=1}^{\floor{kx-1/2}} m\log\bigp{\frac{m^2}{m^2-1/4}}\\
    &= \sum_{m=1}^{\floor{kx-1/2}} -m\log\bigp{1-\frac{1}{4m^2}}
\end{align*}
so since $\Sigma_a, \Sigma_{b_1},\Sigma_{b_2},\dots$ are iid copies of $\Sigma$
\begin{align*}
    nkd_{L^1}(C,\cP) &= \EE \sum_{i=1}^\infty \psi(n\Sigma_i) = \int_0^n \psi(x) \frac{1}{x} dx \ge \sum_{m=1}^{kn-1} -m\log\bigp{1-\frac{1}{4m^2}} \int_{\frac{m+1/2}{k}}^n \frac{1}{x} dx \\
    &\sim \frac{1}{4} \bigp{\log(kn) \sum_{m=1}^{kn-1} \frac{1}{m} - \sum_{m=1}^{kn-1} \frac{\log\bigp{m+1/2}}{m}} \sim \log^2(kn)/8
\end{align*} \qed

Here is an easy extension to general wreath products. Note that $m$ is now fixed so $\cB$ is a finite partition with at most $m$ parts.

\begin{theorem}[Cycle partition convergence]\label{thm:convext}
    Let $\Gamma \subset S_m$ be arbitrary. Take $\pi \in \Gamma^n \rtimes S_n$ uniformly and let $C_1,C_2 \dots$ be the cycle lengths in $\pi$. Let $C = \bigp{\frac{C_1}{mn},\frac{C_2}{mn}, \dots}$. Then
     $$
        d_{L^\infty}(C,\cA) \le 1/n
    $$
    and
    $$
    d_{L^1}(C,\cA) \le \frac{3(\log n + \gamma)}{n} + O(1/n^2)
    $$
    where $\cA$ satisfies
     \begin{equation}\label{stickbreak}
         \cA = (U\cB,(1-U)\cA)
     \end{equation}
     and $\cB = \set{\frac{D_1}{m},\frac{D_2}{m},\dots}$ is the finite random partition distributed according to cycles $D_1,D_2,\dots$ of a uniform permutation from $\Gamma$.
\end{theorem}
\begin{proof}
The exact same proof as Theorem \ref{thm:sqconv} but with $b_1,b_2,\dots$ iid distributed according to $\cB$. Since $\cA = f(\Sigma, \cB, \cB, \dots)$ only the coupling $(a,\Sigma_a)$ is needed.
\end{proof}

As an example consider the generalized symmetric group, known as the group of symmetries of the hypercube and which also happens to be important in understanding which permutations commute with each other \cite{diaconis2024poisson}.

\begin{example}[Generalized symmetric group partition]\label{ex:cyc}
    Consider a random permutation from $C_m^n \rtimes S_n$. Then $\cB$ has equal size parts almost surely and $\pr{\abs{\cB} = \frac{m}{l}} = \frac{\phi(l)}{m}$ for $l | m$, where $\phi(l)$, Euler's totient function, is the number of natural numbers less than $l$ that are coprime to $l$.
\end{example}

\section{Large parts}\label{sec:largeext}

This section derives recursive equations satisfied by the distribution of the $k$th largest part of general self-similar partitions, extending Theorem \ref{thm:sqlarge}. These also include partitions that do not arise as limits of random permutation cycles, and may be of independent interest. Throughout let $\cA_k$ denote the $k$-th largest part of the partition $\cA$. By convention $\cA_j = \infty$ for nonpositive $j$, $\cA_k = 0$ for any $k > \abs{\cA}$, and $1/0 = \infty$.

\begin{theorem}[$k$-th largest part]\label{thm:largerparts}
    Let $\cB$ be a random partition of unity and $\cA$ the random partition of unity defined by
    $$
    \cA = (U\cB,(1-U)\cA)
    $$
    in distribution. Then
    $$
    \phi_k(u) \coloneqq  \pr{\cA_k \le 1/u} = \frac{1}{u}  \sum_{j=0}^{k-1} \EE_{\cB} \int_{(u-1/\cB_{j+1}) \vee 0}^{u-1/\cB_{j}} \phi_{k-j}(y) dy
    $$
    with $\phi_1(u) = 1$ for $u \in [0,1]$. Equivalently
    $$
    u \phi_k(u) = \sum_{j=0}^{k-1} (S_j \ast \phi_{k-j})(u)
    $$
    where $S_j(x) = \pr{\cB_{j+1} \le 1/x < \cB_{j}}$.
\end{theorem}
\begin{proof}
    \begin{align*}
        \pr{\cA_k \le 1/u} &= \pr{(U\cB,(1-U)\cA)_k \le 1/u} = \EE_{\cB} \PP_{U,\cA}\bigp{(U\cB,(1-U)\cA)_k \le 1/u} \\
        &= \EE_{\cB} \int_0^1 \sum_{j=0}^{k-1} 1_{\cB_{j+1} \le \frac{1}{ut} < \cB_j}\pr{\cA_{k-j} \le \frac{1}{u(1-t)}} dt\\
        &= \frac{1}{u} \sum_{j=0}^{k-1} \EE_{\cB} \int_{1/\cB_{j}}^{(1/\cB_{j+1}) \wedge u} \pr{\cA_{k-j} \le \frac{1}{u-y}} dy\\
        &= \frac{1}{u} \sum_{j=0}^{k-1} \EE_{\cB} \int_{(u-1/\cB_{j+1}) \vee 0}^{u-1/\cB_{j}} \phi_{k-j}(y) dy
    \end{align*}

    The second formula follows from Fubini after the third equality and then proceeding similarly.
\end{proof}

A similar approach can be used to derive a differential difference equation satisfied by the joint distribution function of the $k$ largest pieces, but this is not pursued here. In cases where $\cB$ has a bounded number of parts the expectation formula seems more natural. Combining the above with Theorem \ref{thm:convext} gives our most general limit statement for large cycles of random permutations.

\begin{corollary}[Large cycles]
    Let $\Gamma \subset S_m$ be arbitrary. Take $\pi \in \Gamma^n \rtimes S_n$ uniformly and let $C_k$ be the $k$-th largest cycle. Then
    $$
    \frac{C_k}{mn} \dconv \cA_k
    $$
    where $\phi_k(u) \coloneqq  \pr{\cA_k \le 1/u}$ satisfies the recurrence relations of Theorem \ref{thm:largerparts} with $\cB_j$ the $j$-th largest cycle of a uniform permutation from $\Gamma$.
\end{corollary}

 Taking $\Gamma = S_1$, so $\cB_1 = 1$ and $\cB_j = 0$ for $j > 1$ almost surely, recovers $\phi_k = \rho_k$ as in \cite{knuth1976analysis}, where $\rho_1$ is Dickman's function and $\rho_k$ describes the $k$-th largest normalized cycle length of a uniformly random permutation. Thus the above is a strict generalization.

 Here is another continuation of our running Example \ref{ex:cyc}

\begin{example}[$k$-th largest cycle from generalized symmetric group]\label{ex:largepart}
Consider 
\\
$\cB = \set{\frac{l}{m},\frac{l}{m},\dots,\frac{l}{m}}$ with probability $\frac{\phi(l)}{m}$ for $l | m$ ($\phi$ totient). Then the expectation formula gives, for $\psi_k(u) \coloneqq \pr{\cA_k \le 1/u}$
$$
u \psi_k(u) = \sum_{l | m} \frac{\phi(l)}{m} \int_{u-m/l}^u \psi_{k}(t)dt + \sum_{\substack{l | m \\ l > m/k}} \frac{\phi(l)}{m} \int_{0}^{u-m/l}\psi_{k-m/l}(t)dt
$$
Note that when $k=1$ the second term disappears, cleaning up the expression quite a bit. A useful identity when dealing with the totient function is that 
$$
\phi(l) = l\prod_{p | l} \bigp{1-\frac{1}{p}}
$$
Then assuming $u \ge m$ and because $\psi_1$ is nonincreasing
\begin{align*}
    u\psi_1(u) &= \sum_{l | m} \frac{\phi(l)}{m} \int_{ u-m/l}^u \psi_1(t)dt \le \sum_{l | m} \frac{\phi(l)}{m} \frac{m}{l} \psi_1(u-m/l)\\
    &= \sum_{l | m} \prod_{p | l} \bigp{1-\frac{1}{p}} \psi_1(u-m/l) \le m\psi_1(u-m)
\end{align*}
where the last inequality is rather loose. Iterating this $\ceil{u/m}$ times gives
$$
\psi_1(u) \le \frac{m^{\ceil{u/m}}}{u!_{(m)}} = \frac{m^{\ceil{u/m}}}{u(u-m)(u-2m)\dots}
$$
but we conjecture the true decay is much more rapid for large $m$. The term \\ $\sum_{l | m} \frac{\phi(l)}{m} \int_{ u-m/l}^u \psi_1(t)dt$ which determines $\psi_1(u)$ is the expectation over a random integral of the history of $\psi_1$, where the amount of time $m/l$ one looks back is sampled proportional to $\phi(l)$. Our upper bound is then the integral over the most history, rounded to simply be the area of a rectangle. However looking back this far is the lowest probability event, with mass only $1/m$, and this bound gets looser as $m$ grows.
\end{example}

While the above provides an example of bounding large $u$ behavior of these recursive functions they are clearly in general hard to deal with. It is thus natural to ask for explicit identities. The convolution formula for $k=1$ gives $\phi_1(u) = \frac{1}{u} (S_0 \ast \phi_1)(u)$ which is exactly the set of integral equations considered in Section \ref{sec:mf}. Since \cite{granville2001spectrum} actually gives an explicit representation for solutions to such equations we immediately get an expression for the distribution function of the largest part of any self-similar partition. Explicitly, with $S_0 = \psi_1$ the distribution of the largest part $\cB_1$ and $\phi_1$ the distribution of the largest part $\cA_1$ as in Theorem \ref{thm:largerparts} we have
$$
\phi_1(u) = \sum_{k=0}^\infty \frac{(-1)^k}{k!} \int_{\substack{t_1,\dots,t_k \ge 1 \\ t_1 + \dots + t_k \le u}} \frac{1-\psi_1(t_1)}{t_1} \dots \frac{1-\psi_1(t_k)}{t_k}dt_1\dots dt_k
$$
The case of $\phi_1 = \rho$, the Dickman function, is attributed to Ramanujan \cite{andrews2013ramanujan}. The formula is reminiscent of the probability mass function of the number of cycles of a given length in a uniform permutation \cite{goncharov1944some}, and both can be proved with inclusion-exclusion. One may also obtain formulas for $\phi_1$ by way of Fourier inversion of the Laplace transform, which is easy to compute if one knows the Laplace transform for $\psi_1$ (see comments at the end of section \ref{sec:mf}). It is possible some of these formulas could be used to obtain asymptotics.

\appendix

\section{Convergence of cycle counts}\label{app}

Below is a non-normalized version of Theorem \ref{thm:sqconv} with a Compound Poisson limit. The Poisson random variables in the limit exhibit an interesting dependence structure.

\begin{theorem}[Cycle count convergence]\label{thm:nonnormlim}
    Pick $\sigma\in S_k^n \rtimes S_n$ uniformly at random and let $a_i(\sigma)$ for $i \in [kn]$ be the number of cycles of length $i$ in $\sigma$. Then, as both $k,n \to \infty$, the joint distribution of $\{a_i(\sigma)\}_{i=1}^{kn}$ converges (weakly) to the law of $\{A_i\}_{i=1}^\infty$ with 
    $$
     A_i = \sum_{kl = i} \sum_{j=1}^\infty j X_{l,k,j}
    $$
    where $X_{l,k,j} \sim \Poi\bigp{\frac{p^k_j}{l}}$ with $p^{k}_j$ the Poisson PMF with parameter $1/k$ evaluated at $j$. $\set{X_{l,k,j}}_j \cup \set{X_{l',k',j}}_j$ are mutually independent if $l \neq l'$, but otherwise may be dependent. Consequently
    $$
    \EE[A_i] = \frac{\sigma_0(i)}{i}, \quad \Var(A_i) = \frac{\sigma_0(i)}{i} + \frac{1}{i}\sum_{l|i} \frac{1}{l}
    $$
    where $\sigma_0(i)$ is the number of divisors of $i$.
\end{theorem}

First here is a lemma extending Poisson splitting when splitting into countably many processes.

\begin{lemma}[Infinite Poisson splitting]
    Let $\sum_{i=1}^\infty 1_{E_i}$ be distributed $\Poi(\lambda)$ and $Y_i$ iid with support in $\NN_0$ independent of the $E_i$. Then $\set{\sum_{i=1}^\infty 1_{Y_i = j}1_{E_i}}_{j=0}^\infty$ is a countable collection of independent $\Poi(p_j\lambda)$ sums where $p_j = \pr{Y_1 = j}$.
\end{lemma}
\begin{proof}
    Let $X = \sum_{i=1}^\infty 1_{E_i}$ and $X_j = \sum_{i=1}^\infty 1_{Y_i = j}1_{E_i}$. Let $k_1,\dots,k_m \in \NN$ be arbitrary, $k = \sum_{i=1}^m k_i$, and $p = \sum_{j=1}^m p_j$. Then
    \begin{align*}
        \pr{X_1=k_1,\dots,X_m = k_m} &= \sum_{n=k}^\infty \pr{X_1=k_1,\dots,X_m = k_m,X=n} \\
        &= \sum_{n=k}^\infty \frac{e^{-\lambda}\lambda^n}{n!}{n \choose k_1,\dots,k_m, n-k} p_1^{k_1}\dots p_m^{k_m}(1-p)^{n-k}\\
        &= \prod_{j=1}^m \frac{e^{-\lambda p_j}(\lambda p_j)^{k_j}}{k_j!} \sum_{n=k}^\infty \frac{e^{-\lambda(1-p)}(\lambda(1-p))^{n-k}}{(n-k)!}\\
        &= \prod_{j=1}^m \frac{e^{-\lambda p_j}(\lambda p_j)^{k_j}}{k_j!}
    \end{align*}

Since this holds for any $m$ we conclude. Indeed taking $m=1$ gives marginal Poisson distribution and arbitrary $m$ gives independence of the countable collection.
\end{proof}

\begin{proof}[Proof of Theorem \ref{thm:nonnormlim}]
    With notation as in Section \ref{sec:couple}, let $E, F^1, F^2, \dots$ be independent copies of $\Xi$, $E_{k,i}$ the event that there is a $k$-space starting at $i$ in $E$, and $F^i_{k,j}$ the event that there is a $k$-space starting at $j$ in $F^i$. Let $Y^k_i = \sum_{j=1}^\infty 1_{F^i_{k,j}}$, the number of $k$-spaces in $F^i$, which is $\Poi(1/k)$ by Ignatov's theorem \cite{ignatov1982constant}. The coupling and product observation at the start of Section \ref{sec:erdtur} then give that
    $$
    a_i \dconv \sum_{\substack{kl = b}} \sum_{i=1}^\infty Y_i^k 1_{E_{l,i}}
    $$
    so it suffices to show that
    $$
    \sum_{\substack{kl = b}} \sum_{i=1}^\infty Y_i^k 1_{E_{l,i}} \overset{d}{=} \sum_{kl = b} \sum_{j=1}^\infty j X_{l,k,j}
    $$
    Ignatov's theorem also says
    \begin{align}
        \sum_{i=1}^\infty 1_{E_{l,i}} \overset{d}{=} X_l \label{eq:poilimit}
    \end{align}
    where $X_l \sim \Poi(1/l)$ independently. By Fubini we have
    $$
    \sum_{\substack{kl = b}} \sum_{i=1}^\infty Y_i^k 1_{E_{l,i}} = \sum_{\substack{kl = b}} \sum_{j=0}^\infty j \sum_{i=1}^\infty 1_{Y^k_i=j} 1_{E_{l,i}}
    $$
    Now by infinite Poisson splitting
    $$\bigp{\sum_{i=1}^\infty 1_{Y^k_i=j} 1_{E_{l,i}}}_{j=0}^\infty \overset{d}{=} \bigp{X_{l,k,j}}_{j=0}^\infty$$
    and are independent across $l,j$. Indeed the sequences of random variables $\set{Y^k_i}_{i=1}^\infty$ are iid across $i$, so splitting holds. The expectation and variance formulas are straightforward computations using the mean and variance of a Poisson twice each.
\end{proof}

\bibliographystyle{amsplain}
\bibliography{WreathPart.bib}

\end{document}